\documentclass[12pt,leqno]{amsart}

\usepackage[T1]{fontenc}

\usepackage[
  a4paper,
  top=1.2in,
  bottom=1.1in,
  left=1in,
  right=1in,
  headheight=0.5in,
  footskip=0.5in
]{geometry}

\usepackage[none]{hyphenat}  

\usepackage{amssymb}   
\usepackage{mathrsfs}  
\usepackage{accents}   
\usepackage{mathdots}  
\usepackage{nccmath}   
\usepackage{stmaryrd}  
\usepackage{marvosym}  

\usepackage{graphicx}  
\usepackage[dvipsnames]{xcolor}
\usepackage{tikz}     
\usepackage{float}    

\usepackage[format=plain,justification=centering]{caption}
\usepackage{subfig}   

\usepackage[shortlabels]{enumitem}
\setlist[itemize,enumerate]{leftmargin=11mm}

\usepackage{comment}

\usepackage[foot]{amsaddr}
\usepackage[nodayofweek]{datetime}

\usepackage[thmtools-compat]{keytheorems}

\usepackage{cite}
\usepackage{hyperref}
\hypersetup{
  colorlinks=true,
  linkcolor=black,
  urlcolor=Blue,
  citecolor=black
}

\usepackage[capitalise,nameinlink]{cleveref}

\crefname{statement}{Statement}{Statements}
\crefname{conjecture}{Conjecture}{Conjectures}
\crefname{question}{Question}{Questions}
\crefname{observation}{Observation}{Observations}
\crefname{claim}{Claim}{Claims}
\crefname{fact}{Fact}{Facts}
\crefname{axiom}{Axiom}{Axioms}

\newkeytheorem{theorem}[parent=section]

\newkeytheorem{
  statement,lemma,proposition,corollary,conjecture,
  question,observation,claim,fact
}[sibling=theorem,style=plain]

\newkeytheorem{axiom,definition,example}[
  sibling=theorem,
  style=definition
]

\newkeytheorem{remark}[sibling=theorem,style=remark]

\newcommand{\mca}{\mathcal}

\newcommand{\cg}{\text{\sf{cg}}}

\newcounter{tbox}
\crefname{tbox}{Statement}{Statements}

\newcommand{\sta}[1]{%
  \par\medskip\medskip
  \refstepcounter{tbox}%
  \noindent
  \parbox{\linewidth}{(\thetbox) \emph{#1}}%
  \par\vspace*{0.4cm}%
}

\newcommand{\mylongtitle}[1]{%
  \ifodd\value{page}%
    \protect\parbox{0.97\linewidth}{#1}\hfill
  \else
    \hfill\protect\parbox{0.97\linewidth}{#1}%
  \fi
}

\title{Cographs in chordal graphs}

\author{Sepehr Hajebi$^{\ast \dagger}$}
\thanks{$^{\ast}$ Department of Mathematics, University of Toronto, Toronto, Ontario, Canada.}
\thanks{$^{\dagger}$ Supported by the Canada Postdoctoral Research Award (CPRA) program.}
\date{\today}

\begin{document}

\maketitle
\begin{abstract}
For a graph $G$, let $\cg(G)$ be the maximum $k\in \mathbb N$ for which $G$ contains a $k$-vertex cograph as an induced subgraph. The Erd\H{o}s-Hajnal conjecture asserts that for every proper hereditary graph class $\mca{G}$, there exists $\gamma\in(0,1]$ such that every $n$-vertex graph $G\in\mca{G}$ satisfies $\cg(G)\geq n^\gamma$. This is still open; however, for some hereditary classes, one may guarantee much stronger bounds. For instance, every chordal graph $G$ on $n\geq 2$ vertices satisfies $\cg(G)\geq n/(2\log n)$. Paul Seymour asked whether this bound is optimal up to a constant factor. We answer this affirmatively, even for interval graphs: For every integer $n\geq 2$, we construct an $n$-vertex interval graph $G$ with $\cg(G)<64n/\log n$.
\end{abstract}

\section{Introduction}

Graphs in this paper have finite vertex sets, no loops and no parallel edges. We write $\mathbb N$ for the set of all positive integers, and for integers $k,k'$, we denote by $\{k,\ldots,k'\}$ the set of all integers at least $k$ and at most $k'$ (thus, $\{k,\ldots,k'\}=\varnothing$ if and only if $k'<k$). All logarithms are to base $2$. The vertex set and edge set of a graph $G$ are denoted, respectively, by $V(G)$ and $E(G)$. A \textit{clique in $G$} is a set of pairwise adjacent vertices. For $X\subseteq V(G)$, we denote by $G[X]$ the induced subgraph of $G$ with vertex set $X$. A \textit{class} is a set of graphs considered up to isomorphism. A class is \textit{proper} if it does not contain all graphs, and \textit{hereditary} if it is closed under taking induced subgraphs. See \cite{diestel} for standard graph-theoretic terminology.

A \textit{cograph} is a graph with no induced subgraph that is a four-vertex path. Equivalently, cographs are precisely the graphs obtained from single-vertex graphs by iteratively taking disjoint unions and complements; see \cite{CorneilLerchsBurlingham}. For a nonnull graph $G$, let $\cg(G)$ be the largest cardinality of a set $X\subseteq V(G)$ for which $G[X]$ is a cograph. The Erd\H{o}s-Hajnal conjecture \cite{ErdosHajnal} asserts (in one of several equivalent formulations) that
\begin{conjecture}[Erd\H{o}s and Hajnal \cite{ErdosHajnal}]\label{EH}
    For every proper hereditary class $\mca{G}$ of graphs, there exists $\gamma\in(0,1]$ such that every $n$-vertex graph $G\in\mca{G}$ satisfies
$\cg(G)\geq n^\gamma$.
\end{conjecture}

This remains open except for a not impressively long list of classes. On the other hand, for some classes known to satisfy the conjecture, one may in fact guarantee bounds close to linear. A famous example of such a class is that of \textit{chordal} graphs: graphs with no induced subgraph that is a cycle on four or more vertices. It is well known that 
$\cg(G)\geq n/(\lfloor\log n\rfloor+1)\geq n/(2\log n)$ for every chordal graph $G$ on $n\geq 2$ vertices.\footnote{A quick argument: Every $n$-vertex chordal graph has a clique whose deletion leaves components on at most $n/2$ vertices \cite{GilbertRoseEdenbrandt}. Applying this iteratively in the remaining components removes all vertices after at most $\lfloor\log n\rfloor+1$ rounds. The cliques removed in any one round are disjoint with no edges between them, and so their union induces a cograph. Choose a round where the largest number of vertices is removed. Since the rounds together remove all $n$ vertices, this round removes at least $n/(\lfloor\log n\rfloor+1)$ vertices, and these vertices induce a cograph.} 

This bound is special, in that even natural superclasses of chordal graphs need not admit anything close to it. For instance, it is well known \cite{LovaszPerfect} that $\cg(G)\geq n^{1/2}$ for every $n$-vertex perfect graph $G$, and this is best possible up to a constant factor.\footnote{Another quick argument: For each $t\in\mathbb N$, consider the line graph $G$ of the complete bipartite graph $K_{t,t}$. Then $G$ has $n=t^2$ vertices and is perfect \cite{Konig1916}. On the other hand, every induced subgraph of $G$ on $k\in\mathbb N$ vertices with no four-vertex induced path corresponds to a (spanning) subgraph $H$ of $K_{t,t}$ with $k$ edges and no five-vertex (not necessarily induced) path. It follows that every component of $H$ is either a tree or a four-vertex cycle, and so $k=|E(H)|\leq |V(H)|=2t=2n^{1/2}$.}

So, it is natural to ask whether the bound for chordal graphs is also optimal up to constant factors. Paul Seymour mentioned this question to the author at the Barbados Graph Theory Workshop in December 2022. We found recently (in the context of the related topic of ``pure pairs in perfect graphs'') that the answer is indeed ``yes'' even in the smaller class of interval graphs. (Recall that an \textit{interval graph} is the intersection graph of a finite family of intervals of the real line, and that every interval graph is chordal.)
\medskip

The main result of this note is the following:

\begin{restatable}{theorem}{main}\label{thm:main}
    For every integer $n\geq 2$, there exists an $n$-vertex interval graph $G$ with
    $$\cg(G)<\mfrac{64n}{\log n}.$$
\end{restatable}

\section{Proof of \Cref{thm:main}}\label{sec:proof}

A graph $G$ is \textit{outerplanar} if it admits a planar drawing with all vertices on the boundary of the outer face. We say that $G$ is \textit{$2$-degenerate} if every nonnull induced subgraph of $G$ has a vertex of degree at most $2$. We will use the following two standard facts about outerplanar graphs; see \cite{BernhartKainen} for the first and \cite[p.~107]{Harary} for the second:

\begin{theorem}\label{thm:outerplanarnoncrossing}
    A graph $G$ is outerplanar if and only if there is an enumeration $v_1,\ldots,v_n$ of the vertices of $G$ such that for all $i,j,k,l\in\{1,\ldots,n\}$ with $i<j<k<l$, either $v_iv_k\notin E(G)$ or $v_jv_l\notin E(G)$.
\end{theorem}

\begin{theorem}\label{thm:outerplanar2degenerate}
    Every outerplanar graph is $2$-degenerate.
\end{theorem}

We may now prove our main result:

\begin{proof}[Proof of \Cref{thm:main}]
Choose the unique $m\in\mathbb N$ for which $m2^{m-1}<n\leq(m+1)2^m$. Let $M=2^{m+1}$ and let
$$\mca{I}=\left\{[a,a+2^{b-1}):\ a\in\{1,\ldots,2^m\},\ b\in\{1,\ldots,m+1\}\right\}.$$
Each element of $\mca{I}$ is a half-open interval with endpoints in $\{1,\ldots,M\}$ whose length is a nonnegative integer power of $2$. We denote the length of each interval $I\in\mca{I}$ by $\ell(I)$.
\medskip

Let $\Gamma$ be the intersection graph of the intervals in $\mca{I}$. Then $\Gamma$ is an interval graph with $|V(\Gamma)|=|\mca{I}|=(m+1)2^m\geq n$. So it suffices to prove that $\cg(\Gamma)<64n/\log n$; the result would then follow by taking $G$ to be any $n$-vertex induced subgraph of $\Gamma$.
\medskip

Let $\mca{J}\subseteq\mca{I}$ be such that the intersection graph $H=\Gamma[\mca{J}]$ of the intervals in $\mca{J}$ is a cograph. The goal is to prove that $|\mca{J}|<64n/\log n$. Assume that $\mca{J}\neq\varnothing$, and for every $J\in\mca{J}$, let $\mathscr{C}_J$ be the set of all maximal cliques of $H$ containing $J$. We claim that $\{\mathscr{C}_J:J\in\mca{J}\}$ is a ``laminar'' family of sets. Specifically,

\sta{\label{st:laminar} Let $J,J'\in\mca{J}$ be distinct. If $JJ'\notin E(H)$, then $\mathscr{C}_J\cap\mathscr{C}_{J'}=\varnothing$. If $JJ'\in E(H)$, then either $\mathscr{C}_J\subseteq\mathscr{C}_{J'}$ or $\mathscr{C}_{J'}\subseteq\mathscr{C}_J$.}

The first assertion is immediate. For the second assertion, assume that $JJ'\in E(H)$, and suppose for a contradiction that there exist $\mca{C}\in\mathscr{C}_J\setminus\mathscr{C}_{J'}$ and $\mca{C}'\in\mathscr{C}_{J'}\setminus\mathscr{C}_J$. Since $\mca{C}$ is a maximal clique in $H$ with $J'\notin\mca{C}$, it follows that there exists $I\in\mca{C}$ such that $IJ'\notin E(H)$. Likewise, since $\mca{C}'$ is a maximal clique in $H$ with $J\notin\mca{C}'$, it follows that there exists $I'\in\mca{C}'$ such that $I'J\notin E(H)$. In particular, $I,I',J,J'$ are pairwise distinct. But now if $II'\in E(H)$, then $H[\{I,I',J,J'\}]$ is a four-vertex cycle, and if $II'\notin E(H)$, then $H[\{I,I',J,J'\}]$ is a four-vertex path; the former contradicts the fact that $H$ is an interval graph, and the latter contradicts the assumption that $H$ is a cograph. This proves \eqref{st:laminar}.
\medskip

Let $\mathscr{C}=\bigcup_{J\in\mca{J}}\mathscr{C}_J$ be the set of all maximal cliques in $H$. For each $\mca{C}\in\mathscr{C}$, since $\mca{C}$ is a clique in $H$ and the half-open intervals in $\mca{C}\subseteq\mca{J}\subseteq\mca{I}$ have endpoints in $\{1,\ldots,M\}$, it follows that there exists $x_{\mca{C}}\in\{1,\ldots,M-1\}$ such that
$$x^+_{\mca{C}}=x_{\mca{C}}+\mfrac{1}{2}\in\bigcap_{J\in\mca{C}}J.$$
From the maximality of the cliques in $\mathscr{C}$, we deduce that

\sta{\label{st:maximalproperty} For all distinct $\mca{C},\mca{C}'\in\mathscr{C}$, we have $x^+_{\mca{C}}\neq x^+_{\mca{C}'}$. Moreover, for all $J\in\mca{J}$ and every $\mca{C}\in\mathscr{C}$, we have $x_{\mca{C}}^+\in J$ if and only if $\mca{C}\in\mathscr{C}_J$.}

By the first assertion in \eqref{st:maximalproperty}, we may write $\mathscr{C}=\{\mca{C}_1,\ldots,\mca{C}_t\}$, where $t=|\mathscr{C}|\in\mathbb N$, such that
$x^+_{\mca{C}_1}<\cdots<x^+_{\mca{C}_t}.$
Let
$X=\left\{x^+_{\mca{C}_1},\ldots,x^+_{\mca{C}_t}\right\}$, and for each $J\in\mca{J}$, let $X_J=X\cap J$. Then, for every $J\in\mca{J}$, there exist unique $\alpha_J,\beta_J\in\{1,\ldots,t+1\}$ with $\alpha_J<\beta_J$ such that
$$X_J=\{x^+_{\mca{C}_i}:i\in\{\alpha_J,\ldots,\beta_J-1\}\}.$$
Moreover, from \eqref{st:laminar} and the second assertion in \eqref{st:maximalproperty}, it follows that

\sta{\label{st:laminar2} For all distinct $J,J'\in\mca{J}$, either $X_J\cap X_{J'}=\varnothing$, or $X_J\subseteq X_{J'}$, or $X_{J'}\subseteq X_J$.}

Now, recall that $0<x_{\mca{C}_1}<\cdots<x_{\mca{C}_t}<M$. Write $x_{\mca{C}_{0}}=0$ and $x_{\mca{C}_{t+1}}=M$. For each $i\in \{1,\ldots, t+1\}$, let
$S_i=\{x_{\mca{C}_{i-1}}+1,\ldots,x_{\mca{C}_i}\}$. Then $(S_1,\ldots,S_{t+1})$ is a partition of $\{1,\ldots,M\}$. Moreover, for each $J\in\mca{J}$, the left end of $J$ belongs to $S_{\alpha_J}$ and the (open) right end of $J$ belongs to $S_{\beta_J}$.

Let
$$\mca{K}_1=\{J\in\mca{J}:\ell(J)\leq\max\{|S_{\alpha_J}|,|S_{\beta_J}|\}\};$$
$$\mca{K}_2=\{J\in\mca{J}:\ell(J)\geq\max\{|S_{\alpha_J}|,|S_{\beta_J}|\}\}.$$

Then $\mca{J}=\mca{K}_1\cup\mca{K}_2$. We further claim that:

\sta{\label{st:K1small} $|\mca{K}_1|\leq4M$.}

To see this, note that for each $k\in\{1,\ldots,t+1\}$ and every $L\in\{1,\ldots,|S_k|\}$, there are at most $L$ intervals $J\in\mca{J}$ with $\ell(J)=L$ and $\alpha_J=k$, and likewise there are at most $L$ intervals $J\in\mca{J}$ with $\ell(J)=L$ and $\beta_J=k$. Since the intervals in $\mca{J}\subseteq\mca{I}$ have nonnegative integer powers of $2$ as their lengths, it follows that
$$|\mca{K}_1|\leq\sum_{k=1}^{t+1}\sum_{p=0}^{\lfloor\log|S_k|\rfloor}2(2^p)
=2\sum_{k=1}^{t+1}\left(2^{\lfloor\log|S_k|\rfloor+1}-1\right)
\leq4\sum_{k=1}^{t+1}|S_k|=4M.$$
This proves \eqref{st:K1small}.
\medskip

We would also like to prove that $|\mca{K}_2|\leq4M$, but this will take several steps. First, we define $W$ to be the graph with $V(W)=\{1,\ldots,t+1\}$ such that for all distinct $i,j\in\{1,\ldots,t+1\}$ with $i<j$, we have $ij\in E(W)$ if and only if there exists $J\in\mca{J}$ such that $\alpha_J=i$ and $\beta_J=j$. We claim that

\sta{\label{st:Wouterplanar} $W$ is outerplanar.}

Suppose not. Then by \Cref{thm:outerplanarnoncrossing}, there are $i,j,k,l\in\{1,\ldots,t+1\}$ with $i<j<k<l$ such that $ik\in E(W)$ and $jl\in E(W)$. In particular, there are distinct $J,J'\in\mca{J}$ such that
$\alpha_J=i$, $\alpha_{J'}=j$, $\beta_J=k$, and $\beta_{J'}=l$. But now we have
$$x^+_{\mca{C}_i}\in X_J\setminus X_{J'};\qquad
x^+_{\mca{C}_j}\in X_J\cap X_{J'};\qquad
x^+_{\mca{C}_{l-1}}\in X_{J'}\setminus X_J;$$
contrary to \eqref{st:laminar2}. This proves \eqref{st:Wouterplanar}.

\sta{\label{st:sumofS}
$\displaystyle\sum_{\substack{1\leq i<j\leq t+1\\
ij\in E(W)}}\min\{|S_i|,|S_j|\}\leq2M$.}

From \eqref{st:Wouterplanar} and \Cref{thm:outerplanar2degenerate}, it follows that $W$ is $2$-degenerate. In particular, there is an orientation of the edges of $W$ such that in the resulting digraph $D$, every vertex has outdegree at most $2$. Therefore,
$$\sum_{\substack{1\leq i<j\leq t+1\\ ij\in E(W)}}\min\{|S_i|,|S_j|\}
=\sum_{i=1}^{t+1}\sum_{\substack{1\leq j\leq t+1\\
(i,j)\in E(D)}}\min\{|S_i|,|S_j|\}
\leq\sum_{i=1}^{t+1}2|S_i|=2M.$$
This proves \eqref{st:sumofS}.
\medskip

Let us now show that

\sta{\label{st:K2small} $|\mca{K}_2|\leq4M$.}

By \eqref{st:sumofS}, it suffices to prove that for each $ij\in E(W)$ where $i,j\in\{1,\ldots,t+1\}$ with $i<j$, we have $|\{J\in\mca{K}_2:\alpha_J=i,\ \beta_J=j\}|
\leq2\min\{|S_i|,|S_j|\}$. In turn, since for each fixed length, either endpoint in $S_i$ or $S_j$ determines the interval, it suffices to show that
$|\{\ell(J):J\in\mca{K}_2,\ \alpha_J=i,\ \beta_J=j\}|\leq 2$. To see this, suppose for a contradiction that there are distinct intervals $J_1,J_2,J_3\in\mca{K}_2$ with $\ell(J_1)<\ell(J_2)<\ell(J_3)$ such that $\alpha_{J_1}=\alpha_{J_2}=\alpha_{J_3}=i$ and $\beta_{J_1}=\beta_{J_2}=\beta_{J_3}=j$. Since the $\ell(J_1),\ell(J_2),\ell(J_3)$ are all integer powers of $2$, it follows that $\ell(J_3)\geq4\ell(J_1)$, which, combined with $J_1\in\mca{K}_2$, yields
$$\ell(J_3)-\ell(J_1)\geq3\ell(J_1)\geq3\max\{|S_i|,|S_j|\}.$$
On the other hand, since $\alpha_{J_3}=i$ and $\beta_{J_3}=j$, it follows that $\ell(J_3)\leq\max S_j-\min S_i$, and since $\alpha_{J_1}=i$ and $\beta_{J_1}=j$, it follows that
$\ell(J_1)\geq\min S_j-\max S_i$. But then
$$\ell(J_3)-\ell(J_1)\leq(\max S_i-\min S_i)+(\max S_j-\min S_j)<|S_i|+|S_j|\leq 2\max\{|S_i|,|S_j|\};$$
a contradiction. This proves \eqref{st:K2small}.
\medskip

Finally, from \eqref{st:K1small} and \eqref{st:K2small}, it follows that
$|\mca{J}|\leq|\mca{K}_1|+|\mca{K}_2|\leq8M=2^{m+4}$. Since $m2^{m-1}<n$, it follows that $|\mca{J}|\leq 2^{m+4}<32n/m$.
Moreover, since $n\leq(m+1)2^m$ and $m+1\leq2^m$, we have $n\leq2^{2m}$, and so $m\geq(1/2)\log n$. Hence
$$|\mca{J}|<\mfrac{32n}{m}\leq\mfrac{64n}{\log n}.$$
This completes the proof of \Cref{thm:main}.
\end{proof}

\bibliographystyle{abbrv}
\bibliography{ref}
\end{document}